\documentclass[letterpaper,12pt]{article}

\usepackage{amsmath}
\usepackage{amssymb}
\usepackage{amsthm}  
\usepackage[margin=0.8in,letterpaper]{geometry}    
\usepackage{cite} 
\usepackage{caption}
\usepackage{float}
\usepackage{tikz}
\usepackage{authblk}
\usepackage[colorlinks=true,linkcolor=blue,citecolor=blue,urlcolor=blue]{hyperref}
\usepackage{parskip}

\theoremstyle{plain}
\newtheorem{theorem}{Theorem}
\newtheorem{corollary}{Corollary}
\newtheorem{lemma}{Lemma}
\newtheorem{proposition}{Proposition}

\theoremstyle{definition}

\newtheorem{example}{Example}

\theoremstyle{remark}

\DeclareRobustCommand{\seqnum}[1]{%
  \ifmmode
    \text{\href{https://oeis.org/#1}{\textcolor{blue}{{\normalfont\ttfamily #1}}}}%
  \else
    \href{https://oeis.org/#1}{\textcolor{blue}{{\normalfont\ttfamily #1}}}%
  \fi
}

\def\tt{{\mathrm t}}

\author{Andreas M.~Hinz$^{a,}$\footnote{Email:  hinz@math.lmu.de}\; \& El-Mehdi Mehiri$^{b,}$\footnote{Corresponding author. Email: elmehdi.mehiri@emse.fr or  mehiri314@gmail.com}}

\title{\textbf{The Weighted Tower of Hanoi: \\ 
Algebraic Structure, Phase Transitions, \\
and Integer Sequences }\thanks{\textcopyright A.M.Hinz, E.-M.Mehiri 2026}
}

\date{2026--05--15}

\begin{document}

\maketitle

\begin{center}
{\small
$^{a}$ Mathematics Department, LMU M\"unchen, Munich, Germany.\\[4pt]
$^{b}$ Mines Saint-Étienne, CMP, Department of Manufacturing Sciences and Logistics,\\ 
F-13120 Gardanne, France.
}
\end{center}

\begin{abstract}
\noindent We develop a unified algebraic theory of the \emph{weighted Tower of Hanoi} with arbitrary
nonnegative symmetric move costs depending on both disc index and pegs. Starting from a general
optimality recurrence with two competing strategies---one largest-disc move (one-LDM) and
two largest-disc moves (two-LDM)---we derive complete matrix formulations for both regimes
and obtain explicit closed forms for the minimal transfer cost.

 \noindent The one-LDM dynamics is governed by a nontrivial linear operator whose spectral decomposition
reveals a fundamental connection with the \emph{Jacobsthal} and \emph{Lichtenberg} sequences,
while the two-LDM dynamics exhibits pure exponential growth. This framework yields exact
solutions for broad classes of weight models, including peg-symmetric, disc-symmetric,
polynomial, geometric, arithmetic, and sequence-induced costs. In particular, choosing
classical integer sequences (Fibonacci, Lucas, Jacobsthal, Pell, Euler, etc.) as disc weights
produces new derived sequences with explicit formulas and recurrences, establishing the Tower
of Hanoi as a \emph{sequence-generating transform}.

\noindent  We further introduce and analyze models with forbidden moves and move-type-dependent weights,
uncovering a   \emph{phase transition} phenomenon in which the optimal
strategy switches from two-LDM behavior for small discs to one-LDM behavior beyond a finite
threshold. Our results provide a comprehensive
algebraic and combinatorial understanding of weighted Hanoi dynamics and expose deep
connections between optimal solutions   and classical integer sequences.
\end{abstract}



\section{Introduction}

The Tower of Hanoi is one of the most celebrated problems in mathematics and theoretical
computer science, serving simultaneously as a paradigmatic example of recursive algorithms,
a testing ground for combinatorial methods, and a rich source of connections with number
theory and discrete structures.  Since its introduction by Lucas in the nineteenth century \cite{Lucas_1883},
the classical three-peg version has been completely solved, which
yields, for a tower of $n$ discs ({\em $n$-tower}) the well-known minimal
number of moves $2^n-1$, which  is from the sequence \seqnum{A000225} in the
\emph{On-Line Encyclopedia of Integer Sequences}\textsuperscript{\textregistered} (OEIS) \cite{OEIS}.  
For a comprehensive treatment of the history and the mathematical theory of the Tower of Hanoi and its many variants, we refer to the monograph of Hinz, Klav\v{z}ar,  and Petr \cite{Hinz_2018}.

Among the Tower of Hanoi variants, the \emph{weighted Tower of Hanoi} provides a natural   framework
for modeling nonuniform transfer costs \cite{Mehiri_2024}.  In this setting, each move of a disc is assigned a
cost that depends on the disc itself and on the pegs involved in the move, i.e.~the {\em move type}.  

In this work we develop a comprehensive algebraic framework for the weighted three-peg Tower
of Hanoi with arbitrary nonnegative \emph{symmetric} move costs, where symmetry means that
the cost of moving a disc between any two pegs may depend on the disc itself and on the pair of pegs involved, but not on the direction of the move.  Our starting point is a general optimality
recurrence that decomposes every optimal solution into two competing strategies at each
level: a \emph{one largest-disc move} strategy (one-LDM), in which the largest disc moves once,
and a \emph{two largest-disc move} strategy (two-LDM), in which it moves twice.  This
dichotomy leads to two fundamentally different dynamical regimes.

We show that the one-LDM regime is governed by a nontrivial linear operator whose spectral
structure gives rise to deep connections with classical integer sequences, in particular the
Jacobsthal and Lichtenberg sequences.  In contrast, the two-LDM regime exhibits pure
exponential growth and complete decoupling between pegs.  By exploiting this decomposition,
we derive explicit closed forms for the minimal cost of transferring a tower under broad
classes of weight models, including peg-symmetric, disc-symmetric, polynomial, geometric,
arithmetic, and sequence-induced costs.

Another focus of this paper is the study of a phase-transition phenomenon in weighted Hanoi dynamics. For a natural family of nonuniform move-type weights, we show that an initial preference for the two-LDM strategy persists only up to a finite threshold, after which the one-LDM strategy becomes dominant. This reveals a structural mechanism governing several weighted Hanoi problems and suggests broader questions about transition behavior in more general weight models.

Beyond optimal algorithms, our analysis exposes a powerful and previously unnoticed role of
the Tower of Hanoi as a \emph{sequence-generating transform}.  When classical integer
sequences such as the Fibonacci, Lucas, Jacobsthal, or Pell numbers are used as disc weights,
the induced minimal cost sequences admit explicit formulas and satisfy predictable linear
recurrences.  This establishes systematic bridges between combinatorial optimization on
graphs and the theory of integer sequences.

The paper is organized as follows.  In Section~\ref{sec:introduction} we derive the fundamental optimality
recurrence for the weighted Tower of Hanoi.  Sections~\ref{sec:MatrixformulationOneLdm} and~\ref{sec:MatrixformulationtwoLdm} develop the algebraic
structures of the one-LDM and two-LDM regimes.  In Section~\ref{sec:models} we analyze a broad collection of
weight models and their induced integer sequences.  Finally, in Section~\ref{sec:outlook} we discuss perspectives for future work.

\section{General Recurrence}\label{sec:introduction}

 Throughout the paper $n\in\mathbb{N}_0$, if not otherwise stated, discs are labelled by positive integers according to increasing size, and $T:= \{0,1,2\}$ represents the set of pegs. For $i\in T$, the {\em perfect state} with $n$ discs on peg~$i$ is denoted by $i^n$.

For a fixed triple $\{i,j,k\} = T$ and any $n$, we define
\begin{itemize}
    \item $w_{n,k} \in [0,\infty[$ as the cost of a single move of disc $n+1$
    between pegs $i$ and $j$ (so that $k$ is the so-called {\em idle peg} of this move);
    \item $d_{n,k}$ as the minimal total cost (weighted path length in the corresponding {\em Hanoi graph}, whose vertices are the states and with the edges representing individual disc moves) of moving an $n$-tower 
    from state $i^n$ to state $j^n$, respecting the cost of every move according to $w$.
\end{itemize}

The fundamental recurrence is

\begin{theorem}
\label{thm:weighted-hanoi}
Let $\{i,j,k\} = T$. For all $n$, the minimal cost $d_{n+1,k}$ of moving a tower
of $n+1$ discs from $i^{n+1}$ to $j^{n+1}$ satisfies
\begin{equation}\label{eqn:minimum}
   d_{n+1,k} 
   = 
   \min\left\{
       d_{n,i} + d_{n,j} + w_{n,k}, 
       3d_{n,k} + w_{n,i} + w_{n,j}
   \right\},
\end{equation}
with initial condition $d_{0,k}=0$.
\end{theorem}

\begin{proof}
Consider an optimal solution moving
$n+1$ discs from $i^{n+1}$ to $j^{n+1}$. In such a solution, the largest disc $n+1$
can move either once or twice, but not more often (otherwise it would necessarily revisit a peg such that its intermediate moves could be left out, contradicting optimality).

\emph{Case 1 (one largest-disc move).}
The solution has the usual form:
\begin{enumerate}
    \item move the $n$ smaller discs from $i^n$ to $k^n$ using peg $j$ as auxiliary,
    \item move disc $n+1$ from peg $i$ to peg $j$,
    \item move the $n$ smaller discs from $k^n$ to $j^n$ using peg $i$ as auxiliary.
\end{enumerate}
By definition, the first subproblem has cost $d_{n,j}$, the second has cost $w_{n,k}$,
and the third subproblem has cost $d_{n,i}$. Hence the total cost in this case is $d_{n,i} + d_{n,j} + w_{n,k}$.

\emph{Case 2 (two largest-disc moves).}
In this case, the optimal solution has the form:
\begin{enumerate}
    \item move the $n$ smaller discs from $i^n$ to $j^n$ using peg $k$ as auxiliary,
    \item move disc $n+1$ from peg $i$ to peg $k$,
    \item move the $n$ smaller discs from $j^n$ back to $i^n$ using peg $k$ as auxiliary,
    \item move disc $n+1$ from peg $k$ to peg $j$,
    \item move the $n$ smaller discs from $i^n$ to $j^n$ using peg $k$ as auxiliary.
\end{enumerate}
Here each of the three subproblems with $n$ discs has cost $d_{n,k}$, while the two
moves of the largest disc have costs $w_{n,j}$ and $w_{n,i}$, respectively. The total
cost is therefore $3d_{n,k} + w_{n,i} + w_{n,j}$.

Since an optimal solution must realize one of these two patterns, the minimal
cost $d_{n+1,k}$ is the minimum of these two expressions, which gives
\eqref{eqn:minimum}. The initial condition $d_{0,k}=0$ is immediate. This completes
the proof.
\end{proof}

The dichotomy appearing in Theorem~\ref{thm:weighted-hanoi} is closely related to the structure 
of shortest paths  in Hanoi graphs. Hinz~\cite{Hinz_1992}
studied shortest paths between regular states of the Tower of Hanoi and
provided one of the first systematic treatments of the metric structure of
these graphs. Romik~\cite{Romik_2006} later gave a finite-automaton approach
to shortest paths in Hanoi graphs, making explicit the possible recursive
structures of optimal paths. Later, Aumann, G\"otz, Hinz, and
Petr~\cite{Aumann_2014} investigated the number of moves of the
largest disc in shortest paths, a question that is directly connected with
the distinction, used here, between one-LDM and
two-LDM strategies. In the present weighted setting, this
classical structural alternative becomes the two-branch recurrence~\eqref{eqn:minimum}.

Special cases of Theorem~\ref{thm:weighted-hanoi} are
\begin{itemize}
\item classical Tower of Hanoi (Lucas \cite{Lucas_1883}): $w_{n,k}=1$ ,
\item heavy discs (Fried \cite{Fried_2025}, Hinz and Parisse \cite{Hinz_2025}): $w_{n,0}=1=w_{n,2}$ , $w_{n,1}=2$ ,
\item nonmassive discs (Mehiri and Belbachir \cite{Mehiri_2024}): $w_{n,k}=w_{k}$ .
\end{itemize}

In all these cases, the minimum in $\eqref{eqn:minimum}$ is attained by the first entry, i.e.~for one move of disc~$n+1$. So $\eqref{eqn:minimum}$ reduces to
\begin{equation}\label{eqn:minimumonemove}
d_{n+1,k}=d_{n,i}+d_{n,j}+w_{n,k} .
\end{equation}

\section{The one-LDM recurrence}\label{sec:MatrixformulationOneLdm}

In this section we consider the recurrence relation \eqref{eqn:minimumonemove}, corresponding
to the case where the optimal strategy uses exactly one largest-disc move (one LDM) at
each level, and which is purely linear in $(d_{n,0},d_{n,1},d_{n,2})$.

\begin{lemma}
\label{lem:matrix-one-ldm}
Suppose that for all $n$ and all $k\in T$ the one-LDM strategy is
optimal.
Define the following two vectors and matrix, 
$$
d_n = \begin{pmatrix} d_{n,0}\\ d_{n,1}\\ d_{n,2} \end{pmatrix},
\qquad
w_n = \begin{pmatrix} w_{n,0}\\ w_{n,1}\\ w_{n,2} \end{pmatrix},
\qquad  
\mathbf{A}
=
\begin{pmatrix}
0 & 1 & 1\\
1 & 0 & 1\\
1 & 1 & 0
\end{pmatrix}.
$$
Then we have the vector recurrence
\begin{equation}\label{eq:matrix-rec}
    d_{n+1} = \mathbf{A} d_n + w_n.
\end{equation}

\end{lemma}

To solve \eqref{eq:matrix-rec} we need the following

\begin{lemma}
\label{lem:A-power}
For all $n$,
\begin{equation*}\label{eq:A-power-J}
    \mathbf{A}^n
    =
    \begin{pmatrix}
        \widetilde{J}_n & J_n & J_n\\
        J_n & \widetilde{J}_n & J_n\\
        J_n & J_n & \widetilde{J}_n
    \end{pmatrix},
\end{equation*}
where $(J_n)_{n\in\mathbb{N}_0}$ is the \emph{Jacobsthal sequence} \textnormal{(\seqnum{A001045})}, given by $J_n = \frac{1}{3}\bigl(2^n - (-1)^n\bigr)$,  and $\widetilde{J}_n$ is its sequence of forward differences
\textnormal{(\seqnum{A078008})},
$$
\widetilde{J}_n = J_{n+1} - J_n = \frac{1}{3}\bigl(2^n + 2(-1)^n\bigr)=J_n+(-1)^n.
$$
Equivalently,
\begin{equation*}\label{eq:A-power-compact}
    \mathbf{A}^n
    =
    \mathbf{J}_n + (-1)^n \mathbf{I},
\end{equation*}
where $\mathbf{J}_n$ is the constant $3\times 3$ matrix with every entry equal to $J_n$,
and $\mathbf{I}$ is the $3\times 3$ identity matrix.
\end{lemma}

\begin{proof}
This is a straightforward induction. In the induction step calculate $ \mathbf{A}^{n+1}= \mathbf{A}^n\cdot \mathbf{A}$ and
use $J_0=0$, $\widetilde{J}_0=1$,
$J_{n+1}=J_n+\widetilde{J}_n$, and $\widetilde{J}_{n+1}=2J_n$.
\end{proof}

We now solve the linear recurrence \eqref{eq:matrix-rec} in closed form.

\begin{proposition}
\label{prop:solution-linear}
Let $(d_n)_{n\in\mathbb{N}_0}$ be defined by \eqref{eq:matrix-rec}. 
Then for all $n$ we have
\begin{equation}\label{eq:solution-linear}
    d_n
    = \mathbf{A}^n d_0+\sum_{\nu=0}^{n-1} \mathbf{A}^{\nu} w_{n-\nu-1}.
\end{equation}
\end{proposition}
\noindent The {\em proof} is again by induction on $n$.

If the one-LDM regime is optimal from the very beginning, as is often the case, then $d_0=(0,0,0)^\tt$ and the first summand in \eqref{eq:solution-linear} disappears.

\section{The two-LDM recurrence}\label{sec:MatrixformulationtwoLdm}

In Theorem~\ref{thm:weighted-hanoi}, the second branch of the recurrence corresponds
to solutions in which the largest disc moves \emph{twice} at level $n$, yielding
\begin{equation}\label{eq:minimumtwomoves}
d_{n+1,k}
=
3 d_{n,k} + w_{n,i} + w_{n,j}.
\end{equation}
Unlike the one-LDM regime, this system contains no cross-coupling between the three
coordinates. In matrix form, \eqref{eq:minimumtwomoves} is equivalent to the linear recurrence
\begin{equation}\label{eq:two-ldm-matrix}
    d_{n+1}
    =
    3 \mathbf{I}  d_n + b_n,
\end{equation}
where $b_n=|w_n| (1,1,1)^\tt - w_n$, $|w_n|=w_{n,0}+w_{n,1}+w_{n,2}$.

As before, iteration yields:

\begin{proposition}
\label{prop:two-ldm-solution}
Let $(d_n)_{n\in\mathbb{N}_0}$ be defined by \eqref{eq:two-ldm-matrix}. 
Then for all $n$ we have
\begin{equation*}\label{eq:two-ldm-solution}
    d_n
    =
    3^nd_0+\sum_{\nu=0}^{n-1}
        3^\nu  b_{ n-1-\nu}.
\end{equation*}
\end{proposition}

\section{Weight Models}\label{sec:models}

The main differences in largest disc behavior occur for weights depending on (massive) discs or pegs or both.

\subsection{Nonmassive Discs}\label{sec:non-massive}
In this section we assume that the weight does not depend on the disc index, i.e. $w_{n,k}=w_{0,k}=:w_k$ for all $n$ and for all $k\in T$.

\begin{theorem}
\label{thm:non-massive}
Let $w=(w_0,w_1,w_2)^{\mathrm t}$ be independent of $n$, and suppose that at each level
the one-LDM strategy is optimal, so that
\begin{eqnarray*}
    d_0=(0,0,0)^{\mathrm t},\qquad d_{n+1} = \mathbf{A}d_n + w.
\end{eqnarray*}
Then, for all $n$,
\begin{equation}\label{eq:solution-independent}
    d_n
    =
    |w| \ell_{n-1} (1,1,1)^{\mathrm t} + c_n w,
\end{equation}
where $\ell_n$ is defined  (also for $n=-1$) by $\ell_n = \displaystyle{\sum_{\nu=0}^{n}} J_\nu$,  forming the \emph{Lichtenberg sequence} \textnormal{(\seqnum{A000975})} \cite{Hinz_2017},
and $(c_n)_{n\in\mathbb{N}_0}$ is the \emph{parity sequence}
\textnormal{(\seqnum{A000035})} given by $c_n = n \bmod 2$.
\end{theorem}

\begin{proof}
By Proposition~\ref{prop:solution-linear} and the assumption $w_{n,k}=w_k$ for all $n$ and $k$ we
have
$$
d_n
= \sum_{\nu=0}^{n-1} \mathbf{A}^{\nu} w.
$$
Using Lemma~\ref{lem:A-power}, we can write
$$
\mathbf{A}^{\nu}
= \mathbf{J}_\nu + (-1)^{\nu}\mathbf{I}.
$$
Hence
$$
d_n = \sum_{\nu=0}^{n-1}\bigl(\mathbf{J}_\nu + (-1)^{\nu}\mathbf{I}\bigr)w = \sum_{\nu=0}^{n-1}\mathbf{J}_\nu w
   + \sum_{\nu=0}^{n-1}(-1)^{\nu}\mathbf{I} w.
$$

For the first sum, note that $\mathbf{J}_\nu w$ has all three entries equal to
$|w|J_\nu$; thus
$$
\sum_{\nu=0}^{n-1}\mathbf{J}_\nu w
=|w| \left(\sum_{\nu=0}^{n-1} J_\nu\right) (1,1,1)^{\mathrm t}
=|w|  \ell_{n-1} (1,1,1)^{\mathrm t}.
$$
The second sum is 
$$
\left(\sum_{\nu=0}^{n-1}(-1)^{\nu}\right) w = c_nw.
$$
Combining the two contributions, we obtain \eqref{eq:solution-independent}.
\end{proof}

\begin{example}[Classical Tower of Hanoi]\label{ex:classical}
For the classical Tower of Hanoi, every move has unit cost, so $w=(1,1,1)^{\mathrm t}$ and
$|w|=3$. Theorem~\ref{thm:non-massive} gives
$$
d_n = (3\ell_{n-1} + c_n)(1,1,1)^{\mathrm t}.
$$
Using the identity (see \cite[Eq.~(1)]{Hinz_2025})
$$
3\ell_{n-1} + c_n = 2^n - 1,
$$
we recover the well-known Mersenne numbers $M_n=2^n-1$
\textnormal{(\seqnum{A000225})}, i.e.
$$
d_{n,0} = d_{n,1} = d_{n,2} = M_n.
$$
\end{example}

\begin{example}[Heavy middle peg]\label{ex:fried}
Consider $w=(1,2,1)^{\mathrm t}$, so that moves between pegs 0 and 2 have double cost. (This is the variant studied in \cite{Fried_2025, Hinz_2025}.)  Then $|w|=4$, and
Theorem~\ref{thm:non-massive} yields
$$
d_{n,0} = d_{n,2} = 4\ell_{n-1} + c_n,\qquad
d_{n,1} = 4\ell_{n-1} + 2c_n.
$$
Using the identity (see \cite[Eq.~(L.3)]{Hinz_2017})
$$
2\ell_{n-1} + c_n = \ell_n,
$$
we can also write
$$
d_{n,0} = d_{n,2} = 2\ell_n - c_n=\ell_{n+1}-1=\seqnum{A084639}(n),\qquad  
d_{n,1} = 2\ell_n=\seqnum{A167030}(n+2).
$$
By \cite[Theorem~1]{Hinz_2025} the one-LDM branch in Theorem~\ref{thm:weighted-hanoi} is always optimal for this
weight vector,
so these formulas indeed give the true minimal costs.  
\end{example}

\begin{example}[Cheap idle peg]
If $w_{n,0}=1=w_{n,2}$ and $w_{n,1}=0$, i.e.\ $w=(1,0,1)^{\mathrm t}$ and $|w|=2$, then
Theorem~\ref{thm:non-massive} gives
$$
d_{n,0} = d_{n,2} = 2\ell_{n-1} + c_n=\ell_n,\qquad
d_{n,1} = 2\ell_{n-1}=\seqnum{A167030}(n+1).
$$
A direct check shows that for this choice of $w$ the one-LDM branch always
yields the minimum in Theorem~\ref{thm:weighted-hanoi}, so the above formulas again
coincide with the optimal values.
\end{example}

\subsection{Massive Discs}

We now consider the case where weights do not depend on the idle peg, but may vary from disc to disc.

\begin{theorem}
\label{thm:disc-dependent}
Assume that there exists a sequence $(\alpha_n)_{n\in\mathbb{N}_0}$ of nonnegative real numbers with $w_{n,k} = \alpha_n$ for all $k\in T$.  Then $d_{n,0}=d_{n,1}=d_{n,2}=:t_n$, and the sequence $(t_n)_{n\in\mathbb{N}_0}$ satisfies
\begin{equation}\label{eq:disc-dependent-rec}
    t_0=0,\qquad 
t_{n+1} = 2t_n + \alpha_n,
\end{equation}
whence
\begin{equation}\label{eq:disc-dep-explicit}
    t_n =  \sum_{\nu=0}^{n-1} 2^{ n-1-\nu} \alpha_{\nu}.
\end{equation}
\end{theorem}

\begin{proof}
Peg symmetry implies $d_{n,0}=d_{n,1}=d_{n,2}=:t_n$. The recurrence
\eqref{eqn:minimum} becomes
$$
t_0=0,
\qquad
t_{n+1}
=
\min\{2t_n + \alpha_n,   3t_n + 2\alpha_n\}.
$$
We have
$$
(3t_n + 2\alpha_n) - (2t_n + \alpha_n) = t_n + \alpha_n \ge 0,
$$
with strict inequality as soon as either $t_n>0$ or $\alpha_n>0$. Hence, $t_{n+1} = 2t_n + \alpha_n$ for all $n$.

By induction we get \eqref{eq:disc-dep-explicit}. 
\end{proof}

\begin{proposition}[Geometric disc costs]
\label{cor:geometric}
Assume peg symmetry as before and
$
\alpha_n = c r^n
$
for some $c,r\in [0,\infty[$.
Then
\begin{eqnarray*}
    t_n
=
c 2^{n-1}\sum_{\nu=0}^{n-1}\Bigl(\frac{r}{2}\Bigr)^\nu
=
\begin{cases}
\displaystyle 
c 2^{n-1} \frac{1-(r/2)^n}{1-r/2}=c   \frac{2^n-r^n}{2-r}, &\textrm{if } r\neq 2,\\[2mm]
c 2^{n-1} n,&\textrm{if } r=2.
\end{cases}
\end{eqnarray*}
\end{proposition}

\begin{proof}
Substitute $\alpha_\nu = c r^\nu$ into \eqref{eq:disc-dep-explicit} and factor out
$c 2^{n-1}$:
$$
t_n
= \sum_{\nu=0}^{n-1} 2^{n-1-\nu} c r^\nu
= c 2^{n-1} \sum_{\nu=0}^{n-1}\Bigl(\frac{r}{2}\Bigr)^\nu.
$$
This finite geometric progression sum leads to the stated closed form.
\end{proof}

\begin{proposition}[Arithmetic disc costs]\label{prop:arithmetic}
Assume peg symmetry and
$
\alpha_n = a + b n
$
for some $a,b\in [0,\infty[$.
Then the minimal cost sequence $(t_n)_{n\ge0}$ satisfies
\begin{equation*}
    t_n
=
\sum_{\nu=0}^{n-1} 2^{ n-1-\nu}(a+b\nu)
=
a(2^n-1)
+
b\bigl(2^n-n-1\bigr).
\end{equation*}
\end{proposition}
\begin{proof}
By Theorem~\ref{thm:disc-dependent},
$$
t_n=\sum_{\nu=0}^{n-1}2^{ n-1-\nu}\alpha_\nu
=\sum_{\nu=0}^{n-1}2^{ n-1-\nu}(a+b\nu).
$$
Using
$$
\sum_{\nu=0}^{n-1}2^{ n-1-\nu}=2^n-1,
\qquad
\sum_{\nu=0}^{n-1}\nu2^{ n-1-\nu}=2^n-n-1,
$$
the stated closed form follows.
\end{proof}

The special case $a=1=b$ of Proposition~\ref{prop:arithmetic} is

\begin{proposition}[Massive discs with natural masses]
\label{prop:massive-unified}
Assume that the cost of
moving disc $n+1$ is $\alpha_n = n+1$ (natural masses). Then, for all
$n$,
\begin{equation}\label{eq:massive-unified-sum}
    t_n
    = \sum_{\nu=0}^{n-1}2^{ n-1-\nu}(\nu+1)
    = \sum_{\nu=0}^{n-1}2^{\nu}(n-\nu)
    = \sum_{\nu=0}^n M_\nu=E_n,
\end{equation}
where 
$E_n=2^{n+1}-n-2$ forms the Euler sequence \textnormal{\seqnum{A000295}}$(n+1)$.
\end{proposition}

In the previous two cases the model was peg-symmetric.  
We now include the following model:

\begin{proposition}[Massive discs with one cheap idle peg]\label{prop:massive-cheap-peg}
Let $w_{n,0}=n+1=w_{n,2}$, and $ w_{n,1}=0$. Then 
$$
d_{n,0}=d_{n,2}= \sum_{\nu=0}^{n}\ell_\nu,\qquad
d_{n,1}=2\sum_{\nu=0}^{n-1}\ell_\nu,
$$
So we obtain the sequence \seqnum{A178420}$(n+1)$ of partial sums of the Lichtenberg sequence.
\end{proposition}

\begin{proof}
Here the weight depends on the disc index \emph{and} is nonuniform across pegs.
Thus Theorem~\ref{thm:disc-dependent} does not apply directly, since
the symmetry $d_{n,0}=d_{n,1}=d_{n,2}$ is broken.

For
$$
\mathbf{A}^{\nu}
= \mathbf{J}_\nu + (-1)^\nu\mathbf{I},
$$
we compute
$$
\mathbf{A}^{\nu} w_{n-\nu-1}
= (n-\nu) (J_{\nu+1},   2J_\nu,   J_{\nu+1})^{\mathrm t},
$$
so that
$$
d_n = \sum_{\nu=0}^{n-1} (n-\nu) (J_{\nu+1},  2J_\nu,  J_{\nu+1})^{\mathrm t}.
$$
Using the identity (induction!)
$$
\sum_{\nu=0}^{n-1}(n-\nu) J_\nu = \sum_{\nu=0}^{n-1}\ell_\nu,
$$
we obtain the desired closed form for $d_n$. Again elementary calculations show that they are indeed optimal, i.e.~the one-LDM scheme is followed throughout. However, for two discs (only) and idle peg~0 or 2 the largest disc may move once {\em or} twice in a minimal solution!
\end{proof}

We now extend the arithmetic family to general polynomial disc costs. We need the following lemma.
\begin{lemma}\label{lem:polynonials}
Let $P\in\mathbb{R}[x]$ be a polynomial of degree $\delta\in\mathbb{N}_0$. Then there is a unique polynomial $Q\in\mathbb{R}[x]$ of degree $\delta$ such that for all $x\in\mathbb{R}$,
\begin{equation}\label{eq:poly-general-form}
2Q(x)-Q(x+1)=P(x).
\end{equation}
Moreover, if $P$ has integer coefficients, then so has $Q$.
\end{lemma}
\begin{proof}
We write $P(x)=\displaystyle{\sum_{m=0}^\delta p_mx^m}$ and, assuming existence, $Q(x)=\displaystyle{\sum_{m=0}^\delta q_mx^m}$. Then, using the binomial theorem for developing $Q(x+1)$, plugging in everything into \eqref{eq:poly-general-form} and comparing coefficients of powers of $x$, we obtain the system
\begin{equation}\label{eq:polysystem}
\forall m\in\{0,1,\ldots,\delta\}:  q_m-\sum_{\mu=m+1}^\delta \binom{\mu}{m}q_\mu=p_m.
\end{equation}
Starting with $m=\delta$, i.e.~$q_\delta=p_\delta$, the system \eqref{eq:polysystem} can uniquely be solved successively. It is also clear that $q_m\in\mathbb{Z}$, if all $p_m\in\mathbb{Z}$.
\end{proof}

\begin{theorem}[Polynomial disc costs]
\label{thm:poly-costs}
In the situation of Theorem~\ref{thm:disc-dependent} let 
$\alpha_n=P(n)$, with a nontrivial polynomial~$P$. Then for all $n$,
\begin{equation}\label{eq:polynomialsolution}
t_n=\sum_{\nu=0}^{n-1} 2^{n-1-\nu}P(\nu)=Q(0)\cdot 2^n-Q(n)
\end{equation}
with the polynomial $Q$ from Lemma~\ref{lem:polynonials}.
\end{theorem}

\begin{proof}
Define
$$
    s_n=Q(0)\cdot 2^n-Q(n).
$$
Then $s_0=Q(0)-Q(0)=0$, and 
$$
\begin{aligned}
    s_{n+1}-2s_n
    &=
    \bigl(Q(0)\cdot 2^{n+1}-Q(n+1)\bigr)
    -
    2\bigl(Q(0)\cdot 2^n-Q(n)\bigr)  \\
    &=
    2Q(n)-Q(n+1) \\
    &= P(n).
\end{aligned}
$$
Thus
$$
    s_0=0, \qquad s_{n+1}=2s_n+P(n).
$$
The sequence $(s_n)$ satisfies the same recurrence relation and the same initial
condition as $(t_n)$. Hence $s_n=t_n$ for all $n$, and therefore
\begin{equation*}
    t_n=Q(0)\cdot 2^n-Q(n). \tag*{\qedhere}
\end{equation*}
\end{proof}

For $P(n)=1$ we have $Q(n)=1$ and recover the classical Tower of Hanoi $t_n=2^n-1=M_n$ of Example~\ref{ex:classical}. If the mass of discs grows linearly, i.e.~$P(n)=n+1$, we get $Q(n)=n+2$ and $t_n=2^{n+1}-n-2=E_n$ as in Proposition~\ref{prop:massive-unified}. A model for massive circular discs of equal thickness is given in the next example.

\begin{example}[Quadratic costs]
Let $\alpha_n=(n+1)^2=P(n)$. Then, as obtained from \eqref{eq:polysystem}, $Q(n)=n^2+4n+6$ and according to \eqref{eq:polynomialsolution}:
$$
t_n = 6\cdot 2^n -n^2-4n-6=0,1,6,21,58,141,\ldots.
$$
This corresponds to \seqnum{A047520}(n).
\end{example}

If we replace the discs by balls of different diameters, we get the following.

\begin{example}[Cubic costs]
Let $\alpha_n=(n+1)^3=P(n)$. Again \eqref{eq:polysystem} yields
$Q(n)=n^3+6n^2+18n+26$ whence from \eqref{eq:polynomialsolution}
we get
$$
t_n=26\cdot 2^n-n^3-6n^2-18n-26=0,1,10,47,158,441,\ldots,
$$
which is \seqnum{A213575}$(n)$.
\end{example}

\subsubsection{Disc Costs Given By Classical Integer Sequences}

Recurrence \eqref{eq:disc-dependent-rec} shows that choosing $\alpha_n$ to constitute a classical integer sequence automatically
generates a new derived sequence of numbers $t_n$ with explicit structure.  This allows the
Tower of Hanoi model to act as a \emph{sequence transformer}.

\begin{lemma}[Sequence-induced costs]
\label{lem:sequence-costs}
Let $(\alpha_n)_{n\in\mathbb{N}_0}$ fulfill a linear recurrence of order at most $\delta\in\mathbb{N}$ with constant coefficients, i.e.~for all $n$,
$$
\alpha_{n+\delta}=b+\sum_{\nu=0}^{\delta-1} a_\nu\alpha_{n+\nu},
$$
with coefficients $(a_0,\ldots,a_{\delta-1})\in \mathbb{R}^\delta$ and $b\in\mathbb{R}$ given and seeds $(\alpha_0,\ldots,\alpha_{\delta-1})\in\mathbb{R}^\delta$. Let $(t_n)_{n\in\mathbb{N}_0}$ satisfy the recurrence
$$
t_0\in\mathbb{R},\qquad t_{n+1}=2t_n+\alpha_n.
$$
Then for $n\leq \delta$:
\begin{equation*}
    t_n=2^n t_0+\sum_{\nu=0}^{n-1}2^{ n-1-\nu}\alpha_\nu.
\end{equation*}
and for all $n$
\begin{equation*}
t_{n+\delta+1}=b+\sum_{\nu=0}^\delta \tau_\nu t_{n+\nu}
\end{equation*}
with $\tau_\nu=a_{\nu-1}-2a_\nu$, where $a_{-1}:=0$ and $a_\delta:=-1$.
\end{lemma}

\begin{proof}
For $n\leq\delta$, the values of $t_n$ can be obtained from the recurrence for sequence $t$ and the seeds of sequence $\alpha$. For every $n$,
\begin{align*}
t_{n+\delta+1} & = 2t_{n+\delta}+\alpha_{n+\delta}\\
& = 2t_{n+\delta}+ b+\sum_{\nu=0}^{\delta-1} a_\nu\alpha_{n+\nu}\\
& =  2t_{n+\delta}+ b+\sum_{\nu=0}^{\delta-1} a_\nu t_{n+\nu+1}-2\sum_{\nu=0}^{\delta-1} a_\nu t_{n+\nu}\\
& = b+(a_{\delta-1}+2)t_{n+\delta}+\sum_{\nu=1}^{\delta-1}(a_{\nu-1}-2a_{\nu})t_{n+\nu}-2a_0t_n\\
& =  b+\sum_{\nu=0}^\delta \tau_\nu t_{n+\nu}. \tag*{\qedhere}
\end{align*} 
\end{proof}

The simplest application of Lemma~\ref{lem:sequence-costs} is for a homogeneous ($b=0$) linear recurrence of order $\delta=1$, e.g., $\alpha_{n+1}=2\alpha_n$ with $\alpha_0=1$, i.e.~$\alpha_n=2^n$, where $a_0=2$. If we start with $t_0=0$, we get $t_1=1$ and $\tau_0=-4$, $\tau_1=4$, i.e.~the order-$2$ recurrence relation $t_{n+2}=4(t_{n+1}-t_n)$. We recover the special case $c=1$, $r=2$  of Proposition~\ref{cor:geometric}, namely the sequence of numbers $t_n=n2^{n-1}$, which is \seqnum{A001787}.

We now apply Lemma~\ref{lem:sequence-costs} to other classical integer sequences, but avoid leading 0s to prevent disc~1 from being gratuitous.

\begin{proposition}[Fibonacci disc costs]
\label{prop:fib-costs}
Assume peg symmetry and set the disc costs to be $\alpha_n = F_{n+1}$ for all $n$,   where $F_n$ are the {\em Fibonacci numbers} given by  $F_0=0$, $F_1=1$ and $F_{n+2}=F_{n+1}+F_n$ (\seqnum{A000045}). Then for all $n$,
\begin{equation*}
    t_n = 2^{n+1} - F_{n+3},
\end{equation*}
which is \seqnum{A008466}(n+1). Moreover, $(t_n)$ satisfies the order-$3$ recurrence
\begin{equation}\label{eq:fib-rec}
t_0=0,\ t_1=1,\ t_2=3,\qquad t_{n+3}=3t_{n+2}-t_{n+1}-2t_n.
\end{equation}
\end{proposition}

\begin{proof}
Define $u_n=2^{n+1}-F_{n+3}$. Then $u_0=0$ and
$$
u_{n+1}-2u_n
=
-  F_{n+4}+2F_{n+3}
=
F_{n+3}-F_{n+2}
=
F_{n+1}
=
\alpha_n.
$$
Hence $u_{n+1}=2u_n+\alpha_n$ with $u_0=0=t_0$, so $u_n=t_n$ for all $n$.

For \eqref{eq:fib-rec}, we have in Lemma~\ref{lem:sequence-costs}, $\delta=2$, $a_0=1=a_1$, $\alpha_0=1=\alpha_1$, $b=0=t_0$. Then $\tau_0=-2$, $\tau_1=-1$, and $\tau_2=3$.
\end{proof}

\begin{proposition}[Lucas disc costs]
\label{prop:lucas-costs}
Assume peg symmetry and set the disc costs to be $\alpha_n = L_n$ for all  $n$, where $L_n$ are the {\em Lucas numbers} given by  $L_0=2$, $L_1=1$ and $L_{n+2}=L_{n+1}+L_n$ (\seqnum{A000032}).  Then for all $n$, 
\begin{equation*}
    t_n = 3\cdot 2^n - L_{n+2}.
\end{equation*}
Moreover, $(t_n)$ satisfies the order-$3$ recurrence
\begin{equation*}\label{eq:lucas-rec}
t_0=0,\ t_1=2,\ t_2=5,\qquad t_{n+3}=3t_{n+2}-t_{n+1}-2t_n.
\end{equation*}
\end{proposition}

\begin{proof}
Define $u_n=3\cdot 2^n-L_{n+2}$. Then $u_0=0$ and
$$
u_{n+1}-2u_n
=
-  L_{n+3}+2L_{n+2}
=
L_{n+2}-L_{n+1}
=
L_n
=
\alpha_n.
$$
Hence $u_{n+1}=2u_n+\alpha_n$ with $u_0=0=t_0$, so $u_n=t_n$ for all $n$.

The only difference to Proposition~\ref{prop:fib-costs} is the value of $\alpha_0$ which is now 2; it only affects the seeds.
\end{proof}

 \begin{proposition}[Jacobsthal disc costs]
\label{prop:jacobsthal-costs}
Assume peg symmetry and set the disc costs to be $\alpha_n = J_{n+1}$ for all  $n$,  where $J_n$ are the {\em Jacobsthal numbers} given by $J_0=0$, $J_1=1$, and $J_{n+2}=J_{n+1}+2J_n$ (see \cite[(J.1)]{Hinz_2017}). Then for all $n$,
\begin{equation*}
    t_n=\textstyle{\frac{1}{3}}\left((n+1) 2^n-J_{n+1}\right),
\end{equation*}
which forms sequence \seqnum{A045883}.
Moreover, $(t_n)$ satisfies the order-$3$ recurrence
\begin{equation}\label{eq:jacobsthal-rec}
t_0=0,\ t_1=1,\ t_2=3,\qquad t_{n+3}=3t_{n+2}-4t_n.
\end{equation}
\end{proposition}

\begin{proof}
Define $u_n=\textstyle{\frac{1}{3}}\left((n+1) 2^{n}-J_{n+1}\right)$. Then $u_0=0$ and
$$
u_{n+1}-2u_n
=
\textstyle{\frac{1}{3}}\left((n+2) 2^{n+1}-J_{n+2}-(n+1) 2^{n+1}+2J_{n+1}\right)
=J_{n+1},
$$
where we have used $2^{n+1}=J_{n+2}+J_{n+1}$, which is a direct consequence of formulas (J.1) and (J.2) in \cite{Hinz_2017}.
Thus $u_{n+1}=2u_n+\alpha_n$ with $u_0=0=t_0$, whence $u_n=t_n$.

To prove \eqref{eq:jacobsthal-rec}, we refer to Lemma~\ref{lem:sequence-costs} with $\delta=2$, $a_0=2$, $a_1=1$, $\alpha_0=1=\alpha_1$, $b=0=t_0$, such that $\tau_0=-4$, $\tau_1=0$, and $\tau_2=3$.
\end{proof}

 \begin{proposition}[Pell disc costs]
\label{prop:pell-costs}
Assume peg symmetry and set the disc costs to be $\alpha_n = P_{n+1}$ for all $n$, 
where $P_n$ are the {\em Pell numbers} given by 
$
P_0=0$, $P_1=1$, and $P_{n+2}=2P_{n+1}+P_n$ (\seqnum{A000129}).
Then for all $n$,
\begin{equation*}\label{eq:pell-closed}
t_n=P_{n+2}-2^{n+1},
\end{equation*}
which forms \seqnum{A094706}.

Moreover, $(t_n)$ satisfies the order-$3$ recurrence
\begin{equation*}\label{eq:pell-rec}
t_0=0,\ t_1=1,\ t_2=4, \qquad t_{n+3}=4t_{n+2}-3t_{n+1}-2t_n.
\end{equation*}
\end{proposition}

\begin{proof} 
Again we start with defining $u_n=P_{n+2}-2^{n+1}$. Then $u_0=0$ and 
$$
u_{n+1}-2u_n=P_{n+3}-2^{n+2}-2P_{n+2}+2^{n+2}=P_{n+1}=\alpha_n,
$$
such that $u_n=t_n$ for all $n$ follows.

In Lemma~\ref{lem:sequence-costs} we put $\delta=2$, $a_0=1$, $a_1=2$, $\alpha_0=1$, $\alpha_1=2$, and $b=0=t_0$, such that $\tau_0=-2$, $\tau_1=-3$, and $\tau_2=4$.
\end{proof}

The simplest application of Lemma~\ref{lem:sequence-costs} for an inhomogeneous ($b\neq 0$) linear recurrence is for order $\delta=1$:
\begin{proposition}[Mersenne disc costs]\label{prop:mersenne_cost}
Assume peg symmetry and set the disc costs to be $\alpha_n = M_{n+1}$. Then for all $n$, $t_n=(n-1)2^n+1$, which constitutes sequence \seqnum{A000337}, satisfying the order-$2$ recurrence 
\begin{eqnarray*}
    t_0=0,\  t_1=1, \qquad  t_{n+2}=4(t_{n+1}-t_n)+1.
\end{eqnarray*}
\end{proposition}
\begin{proof}
For $u_n=(n-1)2^n+1$ we have $u_0=0$ and
$$
u_{n+1}-2u_n=n2^{n+1}+1-(n-1)2^{n+1}-2=M_{n+1}=\alpha_n,
$$
such that $u_n=t_n$.

In Lemma~\ref{lem:sequence-costs} we have $\alpha_0=1=b$, $a_0=2$, $t_0=0$, and we obtain $\tau_0=-4$ and $\tau_1=4$.
\end{proof}

Our last example is for the Lichtenberg sequence which fulfills the inhomogeneous order-$2$  recurrence relation $\ell_{n+2}=\ell_{n+1}+2\ell_n+1$  (see \cite[(L.1)]{Hinz_2017}).
\begin{proposition}[Lichtenberg disc costs]\label{prop:licht_cost}
Assume peg symmetry and set the disc costs to be $\alpha_n = \ell_{n+1}$. Then $t_n$ is \seqnum{A102301}$(n-1)$ and fulfills the order-$3$ recurrence
\begin{eqnarray*}
    t_0=0,\,  t_1=1,\,  t_2=4, \qquad t_{n+3}=3t_{n+2}-4t_n+1.
\end{eqnarray*}
The closed form of $t_n$ is left as an exercise!
\end{proposition}
\begin{proof}
In Lemma~\ref{lem:sequence-costs} we have $\delta=2$, $\alpha_0=1$, $\alpha_1=2$, $a_0=2$, $a_1=1$, $b=1$, $t_0=0$, and we obtain $\tau_0=-4$ and $\tau_1=0$, and $\tau_2=3$.
\end{proof}

 \subsection{Forbidden Moves}\label{sec:forbidden}

Tower of Hanoi variants with forbidden moves have been studied in several works.
For instance, Sapir~\cite{Sapir_2004} considered the three-peg Tower of Hanoi, restricting some move types or their directions
and presented a general algorithm for solving the only five solvable variants thus obtained.

In this section we study variants of the Tower of Hanoi in which certain moves are 
\emph{forbidden}.  
We assume throughout that direct moves between pegs $0$ and $2$ are disallowed, 
which is equivalent to assigning infinite weight to these moves, i.e., $w_{n,1}=\infty$  for all $n$. This variant is called the linear Tower of Hanoi (LTH) \cite[Section 2.3.1]{Hinz_2018} because we imagine the pegs in a row with peg~1 in the middle, thus avoiding ``long'' moves between the extreme pegs~$0$ and $2$.

If two move types are forbidden, the system degenerates to a single-edge graph and only 
one disc can be transferred.  
If all three move types are forbidden, no legal move exists at all.  
Thus the single-forbidden case considered here is the only nontrivial situation.

\begin{proposition}\label{prop:forbidden-recurrence}
Assume that moves between pegs $0$ and $2$ are forbidden, i.e.\ $w_{n,1}=\infty$ for all $n$. Then
\begin{equation}\label{eq:forbidden-one-step}
d_{n+1}=\mathbf{B} d_n+v_n,
\end{equation}
where $\mathbf{B}=\begin{pmatrix}
0&1&1\\
0&3&0\\
1&1&0
\end{pmatrix}$ and  $v_n=(w_{n,0},  w_{n,0}+w_{n,2},  w_{n,2})^{\mathrm t}.$
\end{proposition}

\begin{proof}
We argue by decomposing an optimal transfer of an $(n\!+\!1)$-tower from peg~$i$ to peg~$j$ according to the forced
moves of the largest disc, namely two if $k=1$ and one otherwise, $\{i,j,k\}=T$.

Consider the recurrence for $k=1$.
We have to (i) clear disc $n+1$ by moving the $n$ smaller discs
onto the opposite side, (ii) move disc $n+1$ to peg~$1$, (iii) transfer the $n$-tower to the opposite peg, (iv) move disc~$n+1$ to its destination, and (v) the $n$-tower on top of it.
Hence
$$
d_{n+1,1}=3 d_{n,1}+w_{n,0}+w_{n,2}.
$$
This yields the middle row of $\mathbf{B}$ and the central component of $v_n$.

Now consider the recurrence for $k=0$. We have to (i) transfer an $n$-tower from $i$ to 0, (ii) disc~$n+1$ from $i$ to $j$, and (iii) an $n$-tower from 0 to $j$. Hence
$$
d_{n+1,0}=d_{n,1}+d_{n,2}+w_{n,0}.
$$
By symmetry (interchanging pegs $0$ and $2$) we obtain for $k=2$:
$$
d_{n+1,2}=d_{n,0}+d_{n,1}+w_{n,2}.
$$
Collecting the three identities yields \eqref{eq:forbidden-one-step}.
\end{proof}

In order to solve the recurrence \eqref{eq:forbidden-one-step} explicitly, we iterate it:
\begin{equation}\label{eq:iteration_forbidden}
d_{n}=\mathbf{B}^{n}d_{0}+\sum_{\nu=0}^{n-1}\mathbf{B}^{\nu}v_{n-1-\nu}.
\end{equation}
Thus the explicit computation of $B^{n}$ is required. 

\begin{proposition}\label{prop:Bn-closed}
Let $\mathbf{B}$ be the matrix defined in Proposition~\ref{prop:forbidden-recurrence}.  
Then for every $n$,
\begin{eqnarray*}
    \mathbf{B}^n = 
\begin{pmatrix}
1-c_n & N_n & c_n\\
0     & 3^n & 0\\
c_n   & N_n & 1-c_n
\end{pmatrix},
\end{eqnarray*}
where
$
N_n = \textstyle{\frac{1}{2}}(3^n-1)
$,
forming the sequence \seqnum{A003462}.
\end{proposition}
\noindent The {\em proof} is by induction using $N_{n+1}=3N_n+1$, $N_0=0$.

Note that $N_{n+1}-N_n =3^n$, generating the sequence \seqnum{A000244}.
An immediate application of Proposition~\ref{prop:forbidden-recurrence} is
 
\begin{corollary}[LTH with unit weights]
\label{thm:forbidden-linear}
Assume that $w_{n,0}=1=w_{n,2}$ and $w_{n,1}=\infty$ for all $n$,  
then
$$
d_{n,0}=d_{n,2}=N_n,
\qquad
d_{n,1} = 2N_n = 3^n - 1;
$$
the latter sequence is \seqnum{A024023}.
\end{corollary}

\begin{proof}
By equation~\eqref{eq:iteration_forbidden} (with $d_0=0$ and $v_n=(1,2,1)^\mathrm{t}$) and Proposition~\ref{prop:Bn-closed},
\begin{equation*}
d_n=\sum_{\nu=0}^{n-1} (2N_\nu+1, 2\cdot 3^\nu, 2N_\nu+1)^\mathrm{t}=N_n(1,2,1)^\mathrm{t}. \tag*{\qedhere}
\end{equation*}
\end{proof}

Even more interesting is the following case.
 
\begin{corollary}[LTH with massive discs]
\label{thm:massive-linear}
Let $w_{n,0}=n+1=w_{n,2}$ and $w_{n,1}=\infty$.  
Then
$$
d_{n,1}
= 2 \sum_{\nu=0}^{n}N_\nu,
\qquad
d_{n,0}=d_{n,2}
= \sum_{\nu=0}^{n}N_\nu;
$$
the latter number is \seqnum{A000340}$(n-1)$.
\end{corollary}

\begin{proof}
By equation~\eqref{eq:iteration_forbidden} (with $d_0=0$ and $v_n=(n+1)(1,2,1)^\mathrm{t}$) and Proposition~\ref{prop:Bn-closed},
\begin{equation*}
d_n=\sum_{\nu=0}^{n-1} (n-\nu)3^\nu   (1,2,1)^\mathrm{t} =\sum_{\nu=0}^n N_\nu   (1,2,1)^\mathrm{t},
\end{equation*}
the latter by induction (cf.~\eqref{eq:massive-unified-sum}).
\end{proof}

\subsection{Other Weight Models}\label{sec:observations}

In this section we analyse models where two pegs have cost $1$ but one peg has cost
$w$. Unlike previous uniform cases, the optimal strategy may switch depending on the
relative size of $w$. This gives rise to a threshold behavior that, as far as we know,
has not been documented in the Hanoi literature.
 
\subsubsection{Constant Nonuniform Weights}
\label{subsec:constweights}

Assume that
\begin{equation}\label{eq:asymmetric_weights}
w_{n,0}=1=w_{n,2},\qquad w_{n,1}=w\in\mathbb{N}_0,
\end{equation}
so that moves with idle peg~1 are more (or less) expensive depending on~$w$.

The following threshold values for $w$ play a decisive role in the analysis of this setting.
Define the sequence $(a_m)_{m\in\mathbb{N}_0}$ by
$$
a_0=0,\qquad a_{m+1}=2\cdot3^m,
$$
so that
$
a_m=0,2,6,18,54,162,\ldots\quad(\text{\seqnum{A008776}}(m-1)).
$

\noindent Numerical experiments suggested that if $a_m\leq w<a_{m+1}$, then (only) for  $n\leq m$ the largest disc moves twice in an optimal solution. This corresponds to the following recurrence relations:
\begin{align}
d_{n+1,0} &= d_{n,0}+d_{n,1}+1,   &\textrm{if }  n\ge0,\label{eq:obs-0}\\
d_{n+1,1} &= 3d_{n,1}+2,          &\textrm{if }  0\le n<m,\label{eq:obs-pre}\\
d_{n+1,1} &= 2d_{n,0}+w,          &\textrm{if }   n\ge m.\label{eq:obs-post}
\end{align}
Moreover, $d_{n,2}=d_{n,0}$ for all $n$ and, as usual, $d_0=(0,0,0)^\mathtt{t}$.

We now show that this will indeed lead to the optimal solutions.
 
\begin{theorem}\label{thm:constant-weight-phase}
Let the weights fulfill \eqref{eq:asymmetric_weights} with 
$a_m\leq w<a_{m+1}$ for some $m\in\mathbb{N}_0$. 
Then
\begin{align}
d_n &= (N_n, 2N_n,N_n)^{\mathrm t}, &\text{if }  0\leq n\leq m,\label{eq:small-n}\\
d_n &=
\begin{pmatrix}
N_m J_{n-m+2}+\ell_{n-m}+w\ell_{n-m-1}\\[1mm]
2N_m J_{n-m+1}+2\ell_{n-m-1}+2w\ell_{n-m-2}+w\\[1mm]
N_m J_{n-m+2}+\ell_{n-m}+w\ell_{n-m-1}
\end{pmatrix}, &\text{if }  n>m.
\label{eq:large-n}
\end{align}
\end{theorem}

\begin{proof} We begin by solving the system of recurrences \eqref{eq:obs-0}, \eqref{eq:obs-pre}, \eqref{eq:obs-post} with $d_0=(0,0,0)^{\mathrm t}$.
By symmetry, $d_{n,0}=d_{n,2}$ for all $n$, so the system
reduces for $d_n=(d_{n,0},d_{n,1})$ and if $n<m$ to 
\begin{equation}\label{eq:matrix_small-n}
d_{n+1}=\mathbf{B}d_n+(1,2)^{\mathrm t}
\end{equation}
with $\mathbf{B}^n=\begin{pmatrix}
1 & N_n\\
0 & 3^n
\end{pmatrix}$
and if $n\geq m$ to
\begin{equation}\label{eq:matrix_large-n}
d_{n+1}=\mathbf{A}d_n+(1,w)^{\mathrm t}
\end{equation}
with $\mathbf{A}^n=
\begin{pmatrix}
J_{n+1} & J_n\\
2J_n & 2J_{n-1}
\end{pmatrix}$ (for $n\geq 1$).
For $n\leq m$ we get from \eqref{eq:matrix_small-n} with initial vector $d_0=(0,0)^{\mathrm t}$:
$$
d_n=\sum_{\nu=0}^{n-1} \mathbf{B}^\nu  (1,2)^{\mathrm t}
$$
and deduce \eqref{eq:small-n}. For $n>m$ we get from \eqref{eq:matrix_large-n} with initial vector $d_m$:
$$
d_n=\mathbf{A}^{n-m}(N_m,2N_m)^{\mathrm t}+\sum_{\nu=0}^{n-m-1} \mathbf{A}^\nu  (1,w)^{\mathrm t}
$$
and deduce \eqref{eq:large-n} after some calculation.

In order to show that \eqref{eq:small-n} and \eqref{eq:large-n} are indeed a solution for \eqref{eqn:minimum}, we have to verify that
\begin{equation}\label{eq:minimum_short}
\forall\;n\in\mathbb{N}_0:\;\min\{d_{n,1},2d_{n,0}+w\}=d_{n,1} ,
\end{equation}
which is indeed the case.

For $d_{n,1}$ we have to and can show with a little effort that
$$
\forall\;n<m:\;\min\{2d_{n,0}+w,3d_{n,1}+2\}=3d_{n,1}+2 ,
$$
and
\begin{equation*}
\forall\;n\geq m:\;\min\{2d_{n,0}+w,3d_{n,1}+2\}=2d_{n,0}+w . \tag*{\qedhere}
\end{equation*}
\end{proof}

Note that
the two expressions in the brackets are only equal if
\begin{itemize}
\item $w=a_0=0(=m)$ and $n=1$,
\item $w=a_{m+1}$ and $n=m$.
\end{itemize}
We therefore have the following situations for an $(n+1)$-tower to be transferred optimally from peg~$0$ to peg~$2$:
\begin{itemize}
\item if $w=0$, then the largest disc is moved once if $n\neq 1$ and once or twice\footnote{``once or twice'' means that there is an optimal solution with one LDM and an optimal solution with two LDMs} otherwise;
\item if $a_m<w<a_{m+1}$, then the largest disc moves twice for $n\leq m-1$ and once for $n\geq m$;
\item if $w=a_{m+1}$, then the largest disc moves twice for $n\leq m-1$, once or twice for $n=m$, and once for $n\geq m+1$.
\end{itemize}

Equality of the two expressions in the brackets in \eqref{eq:minimum_short}, i.e.~for the transfer of an $(n+1)$-tower from peg~$1$ to peg~$2$, say, can only occur if $w=0=n$, such that only for the 1-tower (the only) disc~1 may move once or twice and otherwise the largest disc will always move once only.

 Note that the results of this section are compatible with the above considerations on forbidden moves if you let $w$ tend to infinity.

\subsubsection{Consecutive Weights}
\label{subsec:conweights}

 We now consider the case $w_{n,k}=w+k$ for some $w\in\mathbb{N}_0$ and all $k\in T$. Then
\begin{eqnarray}
d_{n+1,0} & = & \min\{d_{n,1}+d_{n,2},3d_{n,0}+w+3\}+w ,\label{eqn:13}\\
d_{n+1,1} & = & \min\{d_{n,0}+d_{n,2},3d_{n,1}+w+1\}+w+1 ,\label{eqn:14}\\
d_{n+1,2} & = & \min\{d_{n,0}+d_{n,1},3d_{n,2}+w-1\}+w+2 .\label{eqn:15}
\end{eqnarray}
For $w=0$, this reduces to
\begin{eqnarray}
d_{n+1,0} & = & \min\{d_{n,1}+d_{n,2},3d_{n,0}+3\} ,\label{eqn:16}\\
d_{n+1,1} & = & \min\{d_{n,0}+d_{n,2},3d_{n,1}+1\}+1 ,\label{eqn:17}\\
d_{n+1,2} & = & \min\{d_{n,0}+d_{n,1},3d_{n,2}-1\}+2 .\label{eqn:18}
\end{eqnarray}
For $n=0$ this leads to $d_1=(0,1,1)^{\rm t}$, where in \eqref{eqn:18} the second term in the brackets is the minimum. Assuming for the moment that for $n\geq 1$ the minimum in \eqref{eqn:16}, \eqref{eqn:17} and \eqref{eqn:18} is always attained for the first term in the brackets, we obtain
 $$d_{n+1}=\mathbf{A}^{n}d_1+\sum_{\nu=0}^{n-1}\mathbf{A}^{\nu}(0,1,2)^{\rm t}=\mathbf{J_{n}}d_1+(-1)^{n}d_1+\mathbf{L_{n-1}}(0,1,2)^{\rm t}+c_{n}(0,1,2)^{\rm t} ,$$ 
where $\mathbf{A}$ is given, as in Section~\ref{sec:MatrixformulationOneLdm}, by $a_{ij}=1$, $a_{kk}=0$ and $\mathbf{L_{n-1}}$ is the constant $3\times 3$-matrix with element $\ell_{n-1}$. 

The sequences defined by $x_n=d_{n+1,0}$, $y_n=d_{n+1,1}$ and $z_n=d_{n+1,2}$ have interesting properties. We have $x_n=2J_n+3\ell_{n-1}=2\ell_n+\ell_{n-1}=\ell_n+M_n=0,2,5,12,25,52,\ldots$ . This sequence is not in the OEIS. The sequence $y$ is \seqnum{A084170} and fulfills $y_n=x_n+1-c_n=\ell_n+p_{n+1}=\ell_{n+1}+\ell_{n-1}=1,2,6,12,26,52,\ldots$ with the Purkiss sequence $p$ ($p_n=$\seqnum{A051049}$(n-1)$). 
Finally, $z_n=x_n+1=\ell_n+2^n=1,3,6,13,26,53,\ldots$ is \seqnum{A081254}$(n+1)$. It fulfills the recurrence $z_0=1$, $z_{n+1}=2z_n+c_n$, i.e.~the recurrence relation of $\ell$, but with the seed 1 instead of 0. It can easily be checked that the $d_n$ actually fulfill \eqref{eqn:16}, \eqref{eqn:17} and \eqref{eqn:18}.

Now let $w=1$. Numerical evidence shows that we may assume that the first entries in the brackets in \eqref{eqn:13}, \eqref{eqn:14} and \eqref{eqn:15} lead to the respective minima. As before we then get $d_n=(\mathbf{L}_{n-1}+c_n)(1,2,3)^{{\rm t}}$. The sequences $d_{n,0}=6\ell_{n-1}+c_n=M_n+3\ell_{n-1}=0,1,6,13,30,61,126,\ldots$, $d_{n,1}=6\ell_{n-1}+2c_n=2M_n=0,2,6,14,30,62,126,\ldots$, and $d_{n,2}=6\ell_{n-1}+3c_n=3\ell_{n}=0,3,6,15,30,63,126,\ldots$ are  \seqnum{A101622}$(n)$, \seqnum{A000918}$(n+1)$, and \seqnum{A141023}$(n+1)$, respectively. Again it can easily be checked that $d_n$ actually fulfill \eqref{eqn:13}, \eqref{eqn:14} and \eqref{eqn:15}. 

For $w\geq 2$ we only have to add $(w-1)M_n$ to the elements of each sequence, because $M_n$ is the number of disc moves in every optimal solution, and each weight has been augmented by $w-1$.

\section{Outlook}\label{sec:outlook}

The one-LDM regime has turned out to be canonical in most cases of weights throughout this paper. An exception occurred when a move type was forbidden (Section~\ref{sec:forbidden}), where two LDMs were necessary for any number of discs. For an intermediate case, with the weight constant with respect to discs but more important
on one move type than on the others, we observed a phase transition from two-LDM for small numbers of discs to the one-LDM strategy eventually taking over, with the possible mix of the two for specific threshold values of the weight (Section~\ref{subsec:constweights}).

If one allows weights for disc~$n+1$ to grow faster than $3^n$, however, then the opposite transition may occur. This requires highly unbalanced and rapidly growing weights and is
therefore less natural from a modeling perspective.

\begin{example} [Transition from one-LDM to two-LDM]
Let $w_{n,0}=1=w_{n,2}$ and $w_{n,1}=4^n$. Then, for $n\geq 1$, $d_{n,0}=3^{n-1}=d_{n,2}$ (\seqnum{A140429}) and $d_{n,1}=2\cdot 3^{n-1}-1=2a_n-1$ (\seqnum{A048473}$(n-1)$). For all $n$ the solution to transfer an $n$-tower from peg~$1$ to peg~$2$, say, uses one LDM, while for the transfer from $0$ to $2$ this is only so for $n=1$, otherwise two LDMs are necessary.
\end{example}
This can easily be verified from the standard recurrence which here takes the form
\begin{align*}
d_{n+1,0} &= \min\{d_{n,0}+d_{n,1}+1,3d_{n,0}+1+4^n\}\\
          &= \min\{d_{n,1},2d_{n,0}+4^n\}+d_{n,0}+1,\\
d_{n+1,1} &= \min\{2d_{n,0}+4^n,3d_{n,1}+2\}.
\end{align*}

This shows that it would be desirable to find conditions for the weights which will allow to decide on the transition behavior. 

Since our results further reveal that the weighted Tower of Hanoi acts as a
powerful sequence-transformation mechanism, converting classical sequences into new
ones with predictable algebraic structure, more such sequences might be discovered in future investigations.
Another fertile source for integer sequences is the quest for the number of optimal solutions. Whereas in the classical Tower of Hanoi (Example~\ref{ex:classical}) the minimal solution  is unique, it has been shown in \cite[Theorem~2]{Hinz_2025} that for the situation of Example~\ref{ex:fried} there are $2^{J_n}$ shortest solutions to move an $n+1$-tower from peg~$0$ to peg~$1$ and $2^{\widetilde{J}_n}$ for the transfer between pegs~$0$ and $2$.
Numerical experiments and mathematical analysis could also be employed as a rich
algebraic laboratory connecting combinatorics, number theory, and algorithmic dynamics and to
allow for further settings, such as, e.g., negative weights. In any case, our study has demonstrated that challenges from the traditional Tower of Hanoi puzzle do not run out!


\begin{thebibliography}{99}

\bibitem{Aumann_2014}
Aumann, S., G\"otz, K. A. M., Hinz, A. M., \& Petr, C. (2014).
The number of moves of the largest disc in shortest paths on Hanoi graphs.
\emph{The Electronic Journal of Combinatorics, 21}(4), Article P4.38.
\url{https://doi.org/10.37236/4252}

\bibitem{Fried_2025}
Fried, S. (2025).
Solving the Tower of Hanoi puzzle with heavy disks.
\emph{The American Mathematical Monthly, 132}(8), 806.
\url{https://doi.org/10.1080/00029890.2025.2491975}

\bibitem{Hinz_1992}
Hinz, A. M. (1992).
Shortest paths between regular states of the Tower of Hanoi.
\emph{Information Sciences, 63}(1--2), 173--181.
\url{https://doi.org/10.1016/0020-0255(92)90067-I}

\bibitem{Hinz_2017}
Hinz, A. M. (2017).
The Lichtenberg sequence.
\emph{The Fibonacci Quarterly, 55}(1), 2--12.
\url{https://doi.org/10.1080/00150517.2017.12427786}

\bibitem{Hinz_2018}
Hinz, A. M., Klav\v{z}ar, S., \& Petr, C. (2018).
\emph{The Tower of Hanoi: Myths and Maths} (2nd ed.).
Birkh\"auser.
\url{https://doi.org/10.1007/978-3-319-73779-9}

\bibitem{Hinz_2025}
Hinz, A. M., \& Parisse, D. (2025).
The Tower of Hanoi with heavy discs and some integer sequences.
\emph{The Fibonacci Quarterly, 63}(4), 711--716.
\url{https://doi.org/10.1080/00150517.2025.2481593}

\bibitem{Lucas_1883}
Lucas, {\'E}. (1883).
\emph{R\'ecr\'eations math\'ematiques} (Vol.~III).
Gauthier-Villars et fils.

\bibitem{Mehiri_2024}
Mehiri, E.-M., \& Belbachir, H. (2024).
The weighted Tower of Hanoi.
\emph{Discrete Mathematics, Algorithms and Applications, 16}(05), Article 2350051.
\url{https://doi.org/10.1142/S1793830923500519}

\bibitem{OEIS}
OEIS Foundation Inc. (2026).
\emph{The On-Line Encyclopedia of Integer Sequences}.
\url{https://oeis.org}

\bibitem{Romik_2006}
Romik, D. (2006).
Shortest paths in the Tower of Hanoi graph and finite automata.
\emph{SIAM Journal on Discrete Mathematics, 20}(3), 610--622.
\url{https://doi.org/10.1137/050628660}

\bibitem{Sapir_2004}
Sapir, A. (2004).
The Tower of Hanoi with forbidden moves.
\emph{The Computer Journal, 47}(1), 20--24.
\url{https://doi.org/10.1093/comjnl/47.1.20}

\end{thebibliography}
\end{document}